\documentclass[reqno]{amsart}

\usepackage{amssymb,amsfonts,amsmath,amsthm,bm}
\usepackage[mathscr]{eucal}
\usepackage{bbm}
\usepackage{array}
\usepackage{mathdots}
\usepackage{ytableau}

\usepackage{etoolbox}
\patchcmd{\section}{\scshape}{\bfseries}{}{}
\makeatletter
\renewcommand{\@secnumfont}{\bfseries}
\makeatother

\theoremstyle{definition}
\newtheorem{thm}{Theorem}
\newtheorem*{thm*}{Theorem}
\newtheorem{prop}[thm]{Proposition}

\newtheorem{rem}[thm]{Remark}

\numberwithin{equation}{section}

\usepackage{enumerate}
\usepackage[shortlabels]{enumitem}

\usepackage{hyperref}
\newcommand\myshade{85}
\colorlet{mylinkcolor}{red}
\colorlet{mycitecolor}{blue}
\hypersetup{
	linkcolor  = mylinkcolor,
	citecolor  = mycitecolor,
	urlcolor   = myurlcolor!\myshade!black,
	colorlinks = true,
}

\newcommand{\la}[1]{\mathfrak{#1}}

\newcommand{\ZZ}{\mathbb{Z}}

\newcommand{\CC}{\mathbb{C}}

\newcommand{\cW}{\mathcal{W}}

\newcommand{\sW}{\mathscr{W}}

\newcommand{\sL}{\mathscr{L}}

\newcommand{\fS}{\mathfrak{S}}

\DeclareMathOperator{\ch}{\chi}

\newcommand{\weylvec}{\delta}

\allowdisplaybreaks

\begin{document}

\date{}

\title[log-VOAs and torus links]{Characters of logarithmic vertex operator algebras and coloured invariants of torus links}
\author{Shashank Kanade}
\address{Department of Mathematics, University of Denver, Denver, CO 80208}
\email{shashank.kanade@du.edu}
\thanks{We acknowledge the support of Simons Collaboration
 Grant for Mathematicians \#636937. 
 We also thank T.\ Creutzig and K.\ Hikami for their helpful comments.
}

\begin{abstract} 
	We show that the characters of $\mathfrak{sl}_r$ versions of the $(1,p)$ singlet and the $(1,p)$ triplet VOAs arise as limits of appropriately coloured $\mathfrak{sl}_r$ Jones invariants of certain torus links.
\end{abstract}

\dedicatory{
	Dedicated to the loving memory of Aditya Kanade (1996 -- 2023)
}

\maketitle

\section{Introduction}

There exist fascinating connections between various quantum invariants of knots and $3$-manifolds on the one hand and characters of vertex operator algebras (VOAs) on the other.

Most pertinent for our purposes is the work of Morton \cite{Mor-coloured}, where he gave an explicit formula for the coloured $\la{sl}_2$ Jones invariants of the torus knots $T(p,p')$ (where $p,p'$ are coprime), essentially by employing the Rosso--Jones formula \cite{RosJon-torus}. Readers familiar with the representation theory of the Virasoro minimal model VOA will instantly recognize that the large colour limit of Morton's formula given in \cite[Sec.\ 3]{Mor-coloured} is the $q$-character of the simple Virasoro VOA at parameters $p,p'$ (up to some factors and normalizations that do not depend on the the knot.)

Limits of coloured invariants of knots and links have been studied in several works.
We refer the readers to some important works on the existence and explicit computations of limits of $\la{g}$ (where $\la{g}$ is a finite-dimensional simple Lie algebra not necessarily $\la{sl}_2$) coloured invariants of various links -- see \cite{ArmDas-RR}, \cite{BeiOsb}, \cite{GarLe-Nahm},  \cite{GarVuo}, \cite{Haj}, \cite{KeiOsb},  \cite{Yua-q}, \cite{Yua-stability}, \cite{Yua-torus}, etc.

The Virasoro VOA that we have referenced above is nothing but the principal $\sW$-algebra related to the affine VOA based on $\widehat{\la{sl}_2}$ at certain admissible levels. It is now natural to ask if characters of other VOAs may also similarly arise as appropriate limits of quantum invariants of knots and links. Indeed, in a previous article \cite{Kan-torus}, we showed that the characters of the principal $\sW$-algebras based on $\la{sl}_r$ for all $r\geq 2$ arise as limits of $\la{sl}_r$ invariants of torus knots.
These particular VOAs enjoy very special properties. Due to deep results of Arakawa \cite{Ara-c2} and \cite{Ara-princrat}, these are $C_2$-cofinite and rational, and thus their characters are essentially modular invariant.

The present paper is devoted to the study of coloured $\la{sl}_r$ invariants of torus links $T(c,cp)$. Our main theorem asserts that appropriate limits of these invariants give rise to characters of \emph{non-rational} VOAs, namely the $\la{sl}_r$ analogues of the $\la{sl}_2$ $(1,p)$-singlet and the $(1,p)$-triplet VOAs.
The $\la{sl}_2$ versions of the singlet and triplet VOAs are designated as such because of the presence of one (= dimension of Cartan subalgebra in $\la{sl}_2$) distinguished field in the former case and three (= dimension of $\la{sl}_2$) fields in the latter. Thus, the terminology ``singlet'' and ``triplet'' is not exactly appropriate in the higher rank case, however, we keep these names for the lack of better ones.

We consider torus links $T(c,cp)$ and colour each of the $c$ components with the same ``one-rowed'' partition --- $(n)$.  This corresponds to attaching the irreducible module $L_r(n\Lambda_1)$
of $\la{sl}_r$ to each component. 
Recall that $L_r(n\Lambda_1)$ is just the $n$th symmetric power of the defining representation of $\la{sl}_r$.
Then, informally speaking, our main theorems presented in Section \ref{sec:mainthm} state that:
\begin{enumerate}
	\item If $c\leq r$, the $n\rightarrow\infty$ limits of these coloured invariants (after an appropriate shift) are the characters of the $(1,p)$ singlet based on $\la{sl}_c$ (up to some factors depending on $c$ and $r$, but independent of $p$).
	\item If $c= r+1$, the $n\rightarrow\infty$ limits of these coloured invariants (after an appropriate shift) are the characters of the $(1,p)$ triplet VOA based on $\la{sl}_r$, and a family of its distinguished irreducible modules (up to some factors depending on $r$, but independent of $p$). Here, the limits need to be taken with $n$ proceeding along arithmetic progressions of period $r$. 
\end{enumerate}
Our proofs use the formula for these invariants given by Lin and Zheng in \cite{LinZhe}, and further use combinatorial results to match the limits with appropriate VOAs.

Unlike the case of principal $\sW$ algebras that make appearance in relation to the knots $T(p',p)$ \cite{Kan-torus}, the characters of the singlet and triplet VOAs are either not modular invariant (in the former case) or some modifications are necessary (for the latter case). Their analytic properties have been much studied recently, see for instance \cite{CreMil}, \cite{BriMil-singlet}, \cite{BriMil-II}, \cite{BriKasMil}, etc. The structure and representation theory of $\la{sl}_2$ $(1,p)$ singlet and triplet VOAs is by now fairly well-understood \cite{AdaMil-singlet}, \cite{AdaMil-triplet}, \cite{CreMcRYan-ribbonsinglet}, \cite{CreMcRYan-ribbonfullsinglet}. However, results for higher rank versions (and their $(p,p')$ generalizations  \cite{FeiTip}) are few and difficult. For some recent and important progress, see \cite{Sug-feigintipunin}, \cite{Sug-highertriplet}.
We started by asserting that there are fascinating connections between quantum invariants and VOAs. Indeed, the singlet and triplet VOAs have made an appearance in relation to the $\widehat{Z}$ invariants of certain $3$-manifolds in the work \cite{CheEtAl} and \cite{Park}. 
Recently, the torus links $T(2s,2t)$ have been studied in \cite{HikSug} where it is shown that their $\la{sl}_2$ coloured invariants give characters of the logarithmic $(s,t)$-generalizations of the  $\la{sl}_2$ $(1,p)$ singlet.
In conclusion, all of these VOAs are decidedly at the forefront of the present research.

Our primary motivation for the present article is two-fold.
Firstly, it is quite unclear exactly which VOA characters may appear via coloured invariants of links. We believe that building a library of examples will be helpful in understanding this question. 

Secondly, our primary interest is in \emph{using} this bridge to gain more insight on combinatorial properties of the characters. Especially, finding and proving new fermionic formulas for characters of various VOAs has been a very active line of research in the last few decades. Such formulas for various logarithmic VOAs are studied knot-theoretically in \cite{CheEtAl}.
For rational and principal $\sW$ algebras, not much is known in general beyond Virasoro $(p,p')$ and the $\sW_3(3,p')$ models. For Virasoro, see \cite{Welsh}. Even a complete understanding of fermionic formulas for $\sW_3(3,p')$-modules is suprisingly \emph{very} recent; see for instance \cite{KanRus-cylindric} and \cite{War-KRproof}, etc. We hope that knot-theoretic methods are helpful in these investigations.

\section{Schur polynomials and coloured HOMFLY invariants}
\subsection{Partitions} 
A sequence $\mu=(\mu_1,\mu_2,\cdots)$ of non-negative integers with finitely many non-zero entries is called a composition. The length of $\mu$ is denoted by $\ell(\mu)$ and is the location of the right-most non-zero entry. We say that the weight of $\mu$ is $n$ (denoted either by $\mu\vDash n$ or $|\mu|=n$) if $\sum_{i}\mu_i = n$.

We say that $\mu$ is a partition if its entries satisfy $\mu_1\geq \mu_2\geq \mu_3\cdots$. In this case, we say that $\mu\vdash n$ if the weight of $\mu$ is $n$.

We may choose to drop the trailing zero entries from a composition or a partition.
The partition $(n,n,\cdots,n)$ with $c$ occurrences of $n$ will be denoted as $(n)^c$. 

Given a partition $\mu$, its Young diagram consists of left-justified rows of boxes, with row numbers increasing from top to bottom, and with row $i$ having $\mu_i$ many boxes.

A tableau of shape $\mu$ is a filling of the Young diagram of $\mu$ with integers.

A semi-standard Young tableau of shape $\mu$ (SSYT for short) is a tableau of shape $\mu$ where the entries weakly increase across rows but strictly increase down the columns.

Given a partition $\lambda$ and a composition $\mu$, the Kostka number $K_{\lambda,\mu}$ is the number of SSYT of shape $\lambda$ such that the number of occurrences of $i$ equals $\mu_i$ for all $i\in \ZZ_{\geq 1}$.
Clearly, $K_{\lambda,\mu}=0$ if $|\lambda|\neq |\mu|$.

\subsection{Lie algebras $\la{gl}_r$ and $\la{sl}_r$} A partition $\lambda$ with $\ell(\lambda)\leq r$ gives rise to an irreducible finite-dimensional module $\Gamma_r(\lambda)$ for the reductive Lie algebra $\la{gl}_r$. Restricted to the simple Lie algebra $\la{sl}_r$, $\Gamma_r(\lambda)$ remains irreducible, and we will denote the $\la{sl}_r$ module by $L_r(\lambda)$.

Given two partitions $\lambda$, $\lambda'$ with $\ell(\lambda),\ell(\lambda')\leq r$, we have $\Gamma_r(\lambda)\cong \Gamma_r(\lambda')$ iff $\lambda=\lambda'$. However, $L_r(\lambda)\cong L_r(\lambda')$ iff $\lambda_i-\lambda'_i=\lambda_j-\lambda'_j$ for all $1\leq i\leq j\leq r$. In other words, the Young diagrams of $\lambda$ and $\lambda'$ differ by zero more columns of height $r$ on the left.

We choose the Cartan subalgebras of $\la{gl}_r$ and $\la{sl}_r$ to be the subalgebras of diagonal matrices. Let $E_{ij}$ denote the matrix with $1$ in row $i$, column $j$, and $0$s everywhere else.
For $\lambda$ being a partition with at most $r$ parts, $\Gamma_r(\lambda)$ is a direct sum of simultaneous eigenspaces under the action of the Cartan subalgebra. 
The character of $\Gamma_r(\lambda)$  is the generating function for the dimensions of these eigenspaces and it is given by the Schur polynomials:
\begin{align*}
	\ch(\Gamma_r(\lambda))=s_\lambda(x_1,\cdots,x_r)
	=\sum_{\mu\vDash |\lambda|,\ell(\mu)\leq r} K_{\lambda,\mu} x_1^{\mu_1}x_2^{\mu_2}\cdots x_r^{\mu_r},
\end{align*}
where the formal variable $x_i$ keeps track of the eigenvalue with respect to $E_{ii}$.
The character of $L_r(\lambda)$ is just the image of this character in $$\CC[x_1,\cdots,x_r]/\langle x_1x_2\cdots x_r - 1\rangle.$$

Continuing to let $\lambda$ be a partition with $\ell(\lambda)\leq r$, we have
\begin{align}
	\dim(\Gamma_r(\lambda))&=\dim(L_r(\lambda))\\
	&=\sum_{\mu\vDash |\lambda|,\ell(\mu)\leq r} K_{\lambda,\mu}\\
	&= \text{Number\,\,of\,\, SSYT\,\,of\,\,shape}\,\,\lambda\,\,
	\text{with\,\,entries\,\,in\,\,}
	\{1,2,\cdots,r\}.
	\label{eqn:dim}
\end{align}	
Secondly, the weight space of Cartan weight $0$ in $L_r(\lambda)$ is the same as the subspace of $\Gamma_r(\lambda)$ on which all $E_{ii}$ have the same eigenvalue. Thus, we have:
\begin{align}
	\dim&(L_r(\lambda)_0)
	= \text{Number\,\,of\,\, SSYT\,\,of\,\,shape}\,\,\lambda\,\,
	\text{with\,\,entries\,\,in\,\,}
	\{1,2,\cdots,r\}
	\notag\\
	& \text{\,\,where\,\,the\,\,frequencies\,\,of\,\,all\,\,the\,\,entries\,\,}
	{1,\cdots,r}
	\text{\,\,is\,\,same}.
	\label{eqn:dim0}
\end{align}	

We have the following product form whenever $\ell(\lambda)\leq r$:
\begin{align*}
	s_\lambda(q^{r-1},q^{r-3},\cdots,q^{1-r})=
	\prod_{1\leq i< j\leq r}\frac{q^{(\lambda_i-\lambda_j+j-i)}-q^{-(\lambda_i-\lambda_j+j-i)}}{q^{(j-i)}-q^{-(j-i)}}.
\end{align*}
Note that if $\lambda$, $\lambda'$ are partitions such that $\ell(\lambda),\ell(\lambda')\leq r$ and $L_r(\lambda)\cong L_r(\lambda')$, then, $\lambda_i-\lambda_j=\lambda_i'-\lambda_j'$ for all $1\leq i,j\leq r$ and therefore,
\begin{align}
	s_\lambda(q^{r-1},q^{r-3},\cdots,q^{1-r})=
	s_{\lambda'}(q^{r-1},q^{r-3},\cdots,q^{1-r}).
	\label{eqn:princspec}
\end{align}

\subsection{Tensor multiplicities} We recall some basic facts about Schur polynomials.
We have:
\begin{align*}
	s_{(n)}(x_1,\cdots,x_r) = H_n(x_1,\cdots,x_r) = 
	\sum_{r\geq i_1\geq i_2\cdots \geq i_n\geq 1} x_{i_1}x_{i_2}\cdots x_{i_n},
\end{align*}
which is the complete homogeneous symmetric polynomial in variables $x_1,\cdots, x_r$ of total degree $n$.
We further have (see \cite{FulHar})
\begin{align}
	s_{(n)}(x_1,\cdots,x_r)^m = 
	\sum_{\lambda\vdash nm, \ell(\lambda)\leq r}K_{\lambda, (n)^m}s_{\lambda}(x_1,\cdots,x_r).
	\label{eqn:tens_schur}
\end{align}
This relation will play a prominent role in what follows.

Importantly, \eqref{eqn:tens_schur} translates to:
\begin{align}
	\Gamma_r((n))^{\otimes m}&=
		\sum_{\lambda: \lambda\vdash nm, \ell(\lambda)\leq r}K_{\lambda, (n)^m}\Gamma_r(\lambda),	\label{eqn:tens_gl}\\
	L_r((n))^{\otimes m}&=
		\sum_{\lambda: \lambda\vdash nm, \ell(\lambda)\leq r}K_{\lambda, (n)^m}L_r(\lambda).
		\label{eqn:tens_sl}
\end{align}
The second relation is clearly obtained from the first -- 
All partitions appearing in the first sum have the same weight, thus, 
for $\lambda_1$ and $\lambda_2$ appearing in the first sum, $L_r(\lambda_1)\cong L_r(\lambda_2)$ iff $\lambda_1=\lambda_2$. In other words, different summands in \eqref{eqn:tens_gl} restrict to inequivalent $\la{sl}_r$ modules.

\subsection{Coloured invariants of torus links} We now recall from \cite{LinZhe} the formula for coloured HOMFLY invariants of the torus links.

As in \cite{LinZhe}, we define the torus link $T(a,b)$ as the closure of the braid group element $(\sigma_1\sigma_2\cdots\sigma_{a-1})^b$ with $a$ strands. The number of components of $T(a,b)$ equals the gcd of $a$ and $b$. 

Let us consider the link $\sL=T(p'c,pc)$ where $(p',p)=1$. Colour the $c$ components with partitions $\lambda_1,\cdots,\lambda_c$. Suppose that 
for all $1\leq i\leq r$, $\ell(\lambda_i)\leq r$,
and $\lambda_i \vdash k_i$.
Also let $k_1+\cdots+k_c=k$.
Given this data, define the integers $m^{\lambda}_{p';\lambda_1,\cdots,\lambda_c}$ (where $\lambda$ is a partition) by:
\begin{align}
	\prod_{i=1}^c s_{\lambda_i}(x_1^{p'},\cdots,x_r^{p'})=
	\sum_{\lambda\vdash p'k}
	m^{\lambda}_{p';\lambda_1,\cdots,\lambda_c}s_\lambda(x_1,\cdots,x_r)\label{eqn:mlp}.
\end{align}

Given any partition $\lambda$, we need the statistic 
\begin{align*}
	\kappa_\lambda = 2\sum_{i=1}^{\ell(\lambda)}\sum_{j=1}^{\lambda_i} (j-i).
\end{align*}
Note that there is no harm in having the outer summation range over all $i\geq 1$.
If further $\ell(\lambda)\leq r$, define $a_i=\lambda_i-\lambda_{i+1}$ for $1\leq i\leq r$ (in particular, $a_r=\lambda_r$). 
Then, it is not hard to deduce that
\begin{align}
	\kappa_{\lambda}=\sum_{1\leq i\leq j\leq r}\min(i,j)\,a_ia_j - \sum_{1\leq i\leq r}i^2a_i.
	\label{eqn:kappaa}
\end{align}

Up to omission of some insignificant factors, the coloured HOMFLY invariant of $\sL$ is denoted by 
$\cW_{\sL;\, \lambda_1,\cdots,\lambda_c}(t,\nu)$.
This invariant is un-normalized -- for unknot, after specializing $t,\nu$ appropriately, it gives the quantum dimension of the underlying $\la{sl}_r$ modules; see \cite[(2.12), (4.23)]{LinZhe}.
This invariant is given as in the following theorem of \cite{LinZhe}.

\begin{thm}
	Let $\sL$ be the torus link $T(p'c,pc)$ with $(p',p)=1$. Then, we have:
	\begin{align}
		W_{\sL;\, \lambda_1,\cdots,\lambda_c}(t,\nu)
		=t^{\frac{1}{2}pp'\sum_{i}{\kappa_{\lambda_i}}}
		\nu^{\frac{1}{2}p(p'-1)n}
		\sum_{\lambda\vdash p'k} m^{\lambda}_{p';\lambda_1,\cdots,\lambda_c}
		\cdot
		t^{-\frac{p}{2p'}\kappa_{\lambda}}
		\cdot s^{\ast}_\lambda(t,\nu),
		\label{eqn:Winv}
	\end{align}
	where $s^\ast_{\lambda}(t,\nu)$ satisfies (see \cite[(4.25), (5.14)]{LinZhe}):
	\begin{align*}
		s^{\ast}_\lambda(t,\nu)\vert_{t^{-1/2}\mapsto q, \nu^{-1/2}\mapsto q^{r}}=s_{\lambda}(q^{r-1},q^{r-3},\cdots,q^{1-r}).
	\end{align*}
\end{thm}

Very importantly, we will ignore the overall factor of $t,\nu$ on the outside of \eqref{eqn:Winv}, and work with the specialization 
\begin{align}
	t^{-1/2}\mapsto q^{1/2},\quad \nu^{-1/2}\mapsto q^{r/2}.
	\label{eqn:spec}
\end{align}
We shall denote the resulting invariant by $J_r$.
Thus, we have:
	\begin{align}
	J_r(\sL;\, \lambda_1,\cdots,\lambda_c)
	=
	\sum_{\lambda\vdash p'k} m^{\lambda}_{p';\lambda_1,\cdots,\lambda_c}
	\cdot
	q^{\frac{p}{2p'}\kappa_{\lambda}}
	\cdot s_\lambda(q^{\frac{r-1}{2}},q^{\frac{r-3}{2}},\cdots,q^{\frac{1-r}{2}}).
	\label{eqn:Jinv}
\end{align}
We will exclusively work with the case when $\lambda_1=\cdots=\lambda_c=(n)$ and $p'=1$.
In this case, we shall abbreviate the notation as 	$J_r(c,pc;n)$. 
Using \eqref{eqn:tens_schur} and \eqref{eqn:mlp} we have:
\begin{align}
	J_r(c,pc;n)=	\sum_{\substack{\lambda\vdash nc\\ \ell(\lambda)\leq r}} K_{\lambda, (n)^c}
	\cdot
	q^{\frac{p}{2}\kappa_{\lambda}}
	\cdot s_\lambda(q^{\frac{r-1}{2}},q^{\frac{r-3}{2}},\cdots,q^{\frac{1-r}{2}}).
	\label{eqn:ourjones}
\end{align}
Note that due to the SSYT condition, $K_{\lambda,(n)^c}=0$ if $\ell(\lambda)>c$. Thus, the sum is actually over partitions $\lambda$ such that $\lambda\vdash nc$ and $\ell(\lambda)\leq \min(r,c)$.

\section{Logarithmic VOAs}

In this section, we recall relevant information about the $(1,p)$ singlet and triplet VOAs based on the $\la{sl}_r$ root lattices ($r,p\in\ZZ_{\geq 2}$). 
These algebras are denoted as $W^0(p)_{Q_r}$ and $W(p)_{Q_r}$, respectively.
Our main references are
\cite{BriMil-singlet}, \cite{BriMil-II}, and \cite{CreMil}.

\subsection{More notation about $\la{sl}_r$} We first review a bit more notation about the Lie algebra $\la{sl}_r$.

The simple roots will be denoted by $\alpha_i$, $1\leq i\leq r-1$. $Q_r$ will be the root lattice.
The set of positive roots is:
\begin{align*}
	\Phi_r^+=\{ \alpha_i+\alpha_{i+1}+\cdots+\alpha_j\,|\, 1\leq i\leq j\leq r-1\}.
\end{align*}
The fundamental roots will be denoted by $\Lambda_i$, $1\leq i\leq r-1$. $P_r$ will denote the weight lattice and $P_r^+$ the set of dominant integral weights.
$\weylvec$ will be the Weyl vector. 
Note that $P_r/Q_r$ is a cyclic group of order $r$, generated by $\Lambda_1$. Note that we have 
\begin{align}
\Lambda_i - i\Lambda_1\in Q_r,\quad 1\leq i\leq r-1,
\label{eqn:coset}
\end{align}
therefore, sometimes the set of representatives of $P_r/Q_r$ is taken to be $\{ 0, \Lambda_1,\cdots,\Lambda_r\}$ 
(for instance, \cite{BriMil-II}).

$\fS_r$ will denote the Weyl group, which is isomorphic to the symmetric group on $r$ letters. It has a length function $\ell(\cdot)$ such that the simple reflections have length $1$. This length function is not to be confused with the length of compositions.
The usual sign representation of $\fS_r$ is then 
just $(-1)^\ell$.
The (unique) longest element with respect to $\ell$ will be denoted by $w_0$ and satisfies $\ell(w_0)=|\Phi_r^+|$, $w_0(\weylvec)=-\weylvec$.

The lattice $P_r$ has a bilinear, symmetric, $\fS_r$-invariant, non-degenerate form $(\cdot, \cdot)$. Recall that for all $1\leq i\leq j\leq r-1$ we have:
\begin{align}
(\Lambda_i,\Lambda_j) &= \min(i,j)-\frac{ij}{r},
\label{eqn:LambaiLambdaj}\\	
(\weylvec,\alpha_i) &= 1,\\
(\weylvec,\Lambda_i) &= \dfrac{ri-i^2}{2}.
\label{eqn:deltaLambdai}
\end{align}
As usual, $\|\lambda\|^2$ will denote $(\lambda,\lambda)$ whenever $\lambda\in P_r$.

We have seen how partitions of length at most $r$ give rise to irreducible modules $\Gamma_r$ and $L_r$. Actually, the finite-dimensional irreducible modules of $\la{sl}_r$ are indexed by elements of $P_r^+$. Let us therefore review how to translate between the partitions of length at most $r$ and elements of $P_r^+$.

A partition $\mu$ such that $\ell(\mu)\leq r$ corresponds to the element 
\begin{align*}
\underline{\mu}_r=(\mu_1-\mu_2)\Lambda_1+\cdots+(\mu_{r-1}-\mu_{r})\Lambda_{r-1}\in P_r^+.
\end{align*}
Further, using \eqref{eqn:coset} it is not hard to see that
\begin{align*}
	\underline{\mu}_r\in (Q_r+|\mu|\Lambda_1)\cap P_r^+.
\end{align*}
Conversely, a weight $\lambda=a_1\Lambda_1+\cdots+a_{r-1}\Lambda_{r-1}\in P_r^+$ arises from the partition
\begin{align}
	\widehat{\lambda}
	=(a_1+\cdots+a_{r},\,\,a_2+\cdots+a_{r},\,\,\cdots,\,\,a_{r-1}+a_r,\,\,a_r,\,\,0,\,\,\cdots),
	\label{eqn:lambdahat}
\end{align}
where $a_r\in\ZZ_{\geq0}$ can be chosen arbitrarily.
Crucially, using \eqref{eqn:LambaiLambdaj}, \eqref{eqn:deltaLambdai}, and \eqref{eqn:kappaa} we have:
\begin{align}
	(\lambda&,\lambda+2\weylvec)= \sum_{1\leq i\leq j\leq r-1}\left(a_ia_j\min(i,j)-a_ia_j\frac{ij}{r}\right)
	+r\sum_{1\leq i\leq r-1}a_ii - \sum_{1\leq i\leq r-1}i^2a_i\notag\\
	&=\sum_{1\leq i\leq j\leq r}\left(a_ia_j\min(i,j)-a_ia_j\frac{ij}{r}\right)
	+r\sum_{1\leq i\leq r}a_ii - \sum_{1\leq i\leq r}i^2a_i\notag\\
	&=\kappa_{\widehat{\lambda}}+r|\widehat{\lambda}|-\frac{1}{r}|\widehat{\lambda}|^2.
	\label{eqn:thetakappa}
\end{align}
Note that this quantity is independent of $a_r$, since $a_r$ was chosen arbitrarily.

Given the above correspondence between $P_r^+$ and partitions, we have the following formula for Schur polynomials.
Let $\lambda\in P_r^+$ and let $\mu$ be any partition such that $\ell(\mu)\leq r$ and $\underline{\mu}_r=\lambda$. Then,
\begin{align*}
	s_{\mu}&(q^{\frac{r-1}{2}},q^{\frac{r-3}{2}},\cdots,q^{\frac{1-r}{2}})
	=\dfrac{\sum_{w\in\fS_r}(-1)^{\ell(w)} q^{ (w(\lambda + \weylvec),\weylvec) }}
	{\sum_{w\in\fS_r} (-1)^{\ell(w)}q^{ (w\weylvec,\weylvec) }}.
\end{align*}
The fact that the left-hand side is independent of the choice of $\mu$ 
(subject to $\ell(\mu)\leq r$ and $\underline{\mu}_r=\lambda$) is merely a restatement of \eqref{eqn:princspec}. In light of this, 
we may unambiguously define
\begin{align*}
	s_{\lambda}&(q^{\frac{r-1}{2}},q^{\frac{r-3}{2}},\cdots,q^{\frac{1-r}{2}})
\end{align*}
where $\lambda\in P_r^+$ rather than a partition of length at most $r$.

From the Weyl denominator formula, we have:
\begin{align}
	&{\sum_{w\in\fS_r} (-1)^{\ell(w)}q^{ (w\weylvec,\weylvec) }}
	=\prod_{\alpha\in \Phi_r^+}
	(q^{\frac{1}{2}(\weylvec,\alpha)}-q^{-\frac{1}{2}(\weylvec,\alpha)}),\notag\\
	&=(-1)^{|\Phi_r^+|}q^{-\|\weylvec\|^2}\prod_{\alpha\in \Phi_r^+}
	(1-q^{(\weylvec,\alpha)})
	=(-1)^{|\Phi_r^+|}q^{-\|\weylvec\|^2}
	\prod_{1\leq i<j\leq r}
	(1-q^{j-i}),
	\label{eqn:Deltar}
\end{align}
where the first equality follows from the Weyl denominator formula, 
the second by the definition of $\weylvec$, and last by some elementary properties of $\Phi_r^+$.
We are now ready to recall the formulas for various logarithmic VOAs
from \cite{CreMil}, \cite{BriMil-singlet} and \cite{BriMil-II}.

\subsection{$(1,p)$ Singlet}
Let $\eta(q)=q^{1/24}(q)_\infty$. Then, for the $\la{sl}_r$ analogue of the $(1,p)$ singlet VOA,
\begin{align}
	&\eta(q)^{r-1}\ch(W^0(p)_{Q_r})=
	\sum_{\lambda\in Q_r\cap P_r^+}\dim(L_r(\lambda)_0)
	\left(\sum_{w\in \fS_r}(-1)^{\ell(w)}
	q^{\frac{1}{2} \| \sqrt{p} w(\lambda+\weylvec) - \frac{1}{\sqrt{p}}\weylvec \|^2}\right)
	\notag\\
	&=q^{\left(\frac{p}{2}+\frac{1}{2p}\right)\|\weylvec\|^2}\sum_{\lambda\in Q_r\cap P_r^+}\dim(L_r(\lambda)_0)
	q^{\frac{p}{2}(\lambda,\lambda+2\delta)}
	\left(\sum_{w\in\fS_r} (-1)^{\ell(w)} q^{-(w(\lambda+\weylvec),\weylvec)}\right)
	\notag\\
	&=
	q^{\left(\frac{p}{2}+\frac{1}{2p}\right)\|\weylvec\|^2}\sum_{\lambda\in Q_r\cap P_r^+}\dim(L_r(\lambda)_0)
	q^{\frac{p}{2}(\lambda,\lambda+2\delta)}
	\left(\sum_{w\in\fS_r} (-1)^{\ell(w_0^{-1}w)} q^{-(w(\lambda+\weylvec),w_0\weylvec)}\right)
	\notag\\
	&=
	(-1)^{\ell(w_0)}q^{\left(\frac{p}{2}+\frac{1}{2p}\right)
		\|\weylvec\|^2}\sum_{\lambda\in Q_r\cap P_r^+}\dim(L_r(\lambda)_0)
	q^{\frac{p}{2}(\lambda,\lambda+2\delta)}
	\left(\sum_{w\in\fS_r} (-1)^{\ell(w)} q^{(w(\lambda+\weylvec),\weylvec)}\right)
	\notag\\
	&=
	(-1)^{\ell(w_0)}q^{\left(\frac{p}{2}+\frac{1}{2p}\right)
		\|\weylvec\|^2}
	\Delta_r
	\sum_{\lambda\in Q_r\cap P_r^+}\dim(L_r(\lambda)_0)
	q^{\frac{p}{2}(\lambda,\lambda+2\delta)}
	s_{{\lambda}}(q^{\frac{r-1}{2}},q^{\frac{r-3}{2}}\cdots,q^{\frac{1-r}{2}})
	\notag\\
	&=q^{\frac{1}{2p}\|(p-1)\weylvec\|^2}
	\prod_{1\leq i<j \leq r}(1-q^{j-i})\notag\\
	&\quad\cdot
	\sum_{\lambda\in Q_r\cap P_r^+}\dim(L_r(\lambda)_0)
	\cdot
	q^{\frac{p}{2}(\lambda,\lambda+2\delta)}
	\cdot
	s_{{\lambda}}(q^{\frac{r-1}{2}},q^{\frac{r-3}{2}}\cdots,q^{\frac{1-r}{2}}).
\end{align}
We will normalize this character, so that:
\begin{align}
	&\overline{\ch}(W^0(p)_{Q_r})\notag\\
	&=
	\frac{	\prod_{1\leq i<j \leq r}(1-q^{j-i})}
	{(q)_\infty^{r-1}}
	\sum_{\lambda\in Q_r\cap P_r^+}\dim(L_r(\lambda)_0)
	\cdot
	q^{\frac{p}{2}(\lambda,\lambda+2\delta)}
	\cdot
	s_{{\lambda}}(q^{\frac{r-1}{2}},q^{\frac{r-3}{2}}\cdots,q^{\frac{1-r}{2}}).
	\label{eqn:singletchar}
\end{align}

\subsection{$(1,p)$ Triplet}
The $\la{sl}_r$ analogue of the $(1,p)$ triplet has several irreducible modules. We shall require a subset of them.

Let $\{i\Lambda_1\,|\,0\leq i\leq r-1\}$  be the fixed set of representatives of $P_r/Q_r$.
Note again that we have eschewed the choice $\{ 0,\Lambda_1,\cdots,\Lambda_{r-1}\}$ of coset representatives from \cite{BriMil-II}; recall \eqref{eqn:coset}.
Corresponding to each $\gamma$ in this set of representatives, the triplet VOA has an irreducible module $W(p,\sqrt{p}\gamma)_{Q_r}$, such that the VOA itself is $W(p,0)_{Q_r}$, typically denoted simply as $W(p)_{Q_r}$. This VOA has many other irreducible modules which we shall not consider here.
Required characters for us are:
\begin{align}
	\eta&(q)^{r-1}\ch(W(p,\sqrt{p}\gamma)_{Q_r})
	\notag\\
	&=
	\sum_{\mu\in (Q_r+{\gamma})\cap P_r^+}\dim(L_r(\mu))
	\left(\sum_{w\in \fS_r}(-1)^{\ell(w)}
	q^{\frac{1}{2} \| \sqrt{p} w(\mu+\weylvec) - \frac{1}{\sqrt{p}}\weylvec \|^2}\right)\notag\\
	&=q^{\frac{1}{2p}\|(p-1)\weylvec\|^2}
	\prod_{1\leq i<j \leq r}(1-q^{j-i})
	\notag\\
	&\quad\cdot
	\sum_{\mu\in (Q_r+{\gamma})\cap P_r^+}\dim(L_r(\mu))
	\cdot q^{\frac{p}{2}(\mu,\mu+2\delta)}
	\cdot s_{{\mu}}(q^{\frac{r-1}{2}},q^{\frac{r-3}{2}}\cdots,q^{\frac{1-r}{2}}).
\end{align}
Again, we shall normalize so that:
\begin{align}
	\overline{\ch}&(W(p,\sqrt{p}\gamma)_{Q_r})\notag\\
	&=
	\frac{	\prod_{1\leq i<j \leq r}(1-q^{j-i})}
	{(q)_\infty^{r-1}}
	\sum_{\mu\in (Q_r+{\gamma})\cap P_r^+}\dim(L_r(\mu))
	q^{\frac{p}{2}(\mu,\mu+2\delta)}
	s_{{\mu}}(q^{\frac{r-1}{2}},q^{\frac{r-3}{2}}\cdots,q^{\frac{1-r}{2}}).
	\label{eqn:tripletchar}
\end{align}

\section{Main combinatorial theorems}

Characters of the VOAs recalled in the previous section involve dimensions of various $\la{sl}_r$ modules, or their weight $0$ subspaces. On the other hand, the formula \eqref{eqn:ourjones} for $\la{sl}_r$ invariants of torus links involves the Kostka numbers. In this section we deduce the key combinatorial theorems that connect the two.

\begin{prop}
	\label{prop:dim0}
	Let $\lambda$ be a partition such that $\ell(\lambda)\leq r$. 
	If $r\nmid |\lambda|$ (equivalently, if $\underline{\lambda}_r\not\in Q_r\cap P_r^+$) then 
	\begin{align*}
		\dim(L_r(\lambda)_0)=0.
	\end{align*}		
	If $\lambda \vdash rk$ (equivalently, $\underline{\lambda}_r\in Q_r\cap P_r^+$) then we have:
	\begin{align*}
		K_{\lambda,(n)^r} = \dim(L_r(\lambda)_0).
	\end{align*}
\end{prop}
\begin{proof}
	The first claim follows immediately from the fact that the set of weights of $L_r(\lambda)$ is a subset of $(Q_r+\lambda) \cap P_r$.
	Second part is just \eqref{eqn:dim0}.	
\end{proof}

\begin{prop}
	\label{prop:r+1dim}
	Let $\lambda\vdash n(r+1)$ be a partition with $\ell(\lambda)\leq r$. 
	If $\lambda_r\geq n$ then we have
	\begin{align*}
			K_{\lambda,(n)^{r+1}}=\dim(L_r(\lambda)).
	\end{align*}
\end{prop}

\begin{proof}
	Let $\mu$ be the partition obtained by deleting the left-most $\lambda_r$ many columns (which all have height $r$) in the Young diagram of $\lambda$.
	So, $|\mu| = |\lambda|-n\lambda_r \leq  n$.

	To show
	$$K_{\lambda,(n)^{r+1}}=\dim(L_r(\lambda))=\dim(L_r(\mu)),$$ 
	it suffices to show that the number of SSYT of shape $\lambda$ with exactly $n$ copies of each of the integers $1,2,\cdots,r+1$ is the same as the number of SSYT of shape $\mu$ with entries in $1,2,\cdots,r$ (with no restrictions on their frequencies); see \eqref{eqn:dim}.
	Denote the former set of SSYT by $S_\lambda$ and the latter by $S_\mu$. We produce a bijection between these sets.
	
	Let $T_\lambda\in S_\lambda$. Denote the tableau formed by the first $\lambda_r$ columns of $T_\lambda$ by $T_l$.
	Since we must use $n$ many $1$s in $T_\lambda$ and since $\lambda_r\geq n$, the $n$ left-most entries in the top row of $T_l$ (and hence of $T_\lambda$)
	must all be $1$. This also means that no $1$s can now appear outside of $T_l$.
	The portion of $T_\lambda$ to the right of $T_l$ has shape $\mu$. In this tableau, each entry must now be at least $2$. Subtracting $1$ from each entry, we obtain an SSYT $T_\mu\in S_\mu$.
	We will denote this process as: $$T_\lambda = T_l \cup (T_\mu +1),$$
	in particular, we use the notation $T_\mu +d$ to increase each entry of a
	given tableau by $d$.
	
	Our required map is:
	\begin{align*}
		\varphi: S_\lambda &\rightarrow S_\mu\\
		T_\lambda&\mapsto T_\mu,
	\end{align*}
	which we will now show to be a bijection.
	An example with $n=7$, $r=3$, $\lambda = (11,9,8)$, and $\mu=(3,1)$ is as follows: 
	\begin{align*}
		T_\lambda =
		\begin{matrix}
		\begin{ytableau}
				1 & 1 & 1 & 1 & 1 &1 & 1 & 2 & *(gray!25) 3 & *(gray!25) 3 & *(gray!25) 4\\
			2 & 2 & 2 & 2 & 2 &2 & 3 & 3 & *(gray!25) 4 \\
			3 & 3 & 3 & 4 & 4 &4 & 4 & 4
		\end{ytableau}
		\end{matrix}
		\,\,
		\longmapsto 
		\,\,
		T_\mu =
		\begin{matrix}
			\begin{ytableau}
				*(gray!25) 2 & *(gray!25) 2 & *(gray!25) 3\\
				*(gray!25) 3
			\end{ytableau}
		\end{matrix}
	\end{align*}		
	
	First, we argue that $\varphi$ is injective. 
	
	Suppose $\varphi(T_\lambda) = \varphi(T_\lambda')$, i.e., $T_\mu=T_\mu'$.
	The following facts are clear.
	\begin{enumerate}
		\item 	For each $1\leq i\leq r+1$, the number of occurrences of $i$ in $T_l$ is the same as the number of occurrences of $i$ in $T'_l$. 
		\item Due to column strict condition, for all $j$, $j$th rows of $T_l$ and $T'_l$ can only have entries belonging to $\{j,j+1\}$.
	\end{enumerate}		
	Now we have the following inductive argument:		
	\begin{enumerate}
		\item $1$s can only appear at the far left of the first row.  Thus the locations of $1$s in $T_l$ and $T'_l$ match exactly.
		\item If any other number appears in the first row of $T_l$ or $T'_l$, it must now be $2$. Thus, the first rows of  $T_l$ and $T'_l$ match.
		\item Continuing, the second rows can only have numbers $2$ and $3$, with twos appearing to the left of $3$s, and so the second rows of $T_\lambda$ and $T_\lambda'$ must also match.
		\item Continuing by induction, it is now clear that $T_l=T'_l$. 
	\end{enumerate} 
	Since we are working under the assumption that $T_\mu=T_\mu'$, we now have $$T_\lambda=T_l\cup (T_\mu+1)=T'_l\cup (T_\mu'+1)=T_\lambda'.$$
	
	Next, we show that $\varphi$ is surjective.
	
	Let $T_\mu\in S_\mu$ (recall that this means that $T_\mu$ is an SSYT of shape $\mu$ and has entries from $1,2,\cdots, r$). For $1\leq i\leq r+1$, let $n_i$ be the number of times $i$ appears in $T_\mu+1$. Note that $n_1=0$. Clearly, $n\geq |\mu|\geq n_i$, so $n\geq n_i$ for each $n_i$.	
	Start with $n-n_i$ many copies of $i$ for each $1\leq i \leq r+1$. 
	In total, we thus have 
	$$\sum_{1\leq i \leq r+1}(n-n_i) = (r+1)n - \sum_i n_i = (r+1)n - |\mu|=r\lambda_r$$
	many numbers available. 
	
	Consider the tableau $T_l$ with $r$ rows and $\lambda_r$ columns, filled with $n-n_i$ copies of each $1\leq i\leq r+1$ in the following way. Fill these numbers starting from top left, continuing from left to right in each row while traversing the rows top to bottom. First fill in all available copies of $1$. Then use all the $2$s, and continue until all numbers are exhausted.

	We claim that $T_l$ is an SSYT. All we need to show is that the entries strictly increase in columns. Suppose that two copies of some $i$ appear in the same column in $T_l$. A moment's thought shows that in this case $n-n_i>\lambda_r$. This is not possible, since $\lambda_r\geq n$.
	Thus, $T_l$ is a SSYT. In particular, the last entry in $j$th column of $T_l$ is at most $j+1$.
	
	Now we show that this $T_l$ can be fused with $T_{\mu}+1$ to obtain an SSYT of shape $\lambda$.	
	Since $T_\mu+1$ is an SSYT with entries in $\{2,\cdots, r+1\}$, column strictness implies that the first entry in its $j$th row (if the row is non-empty) must be at least $j+1$. This means that when we consider $T_l\cup (T_\mu+1)$,
	at the interface of $j$th row, the entry at the end of $T_l$ is at most $j+1$ and the beginning of $(T_\mu+1)$ is at least $j+1$. Consequently, the entries increase weakly in each row. All in all, $T_l\cup (T_\mu+1) \in S_\lambda$ and it clearly maps to $T_\mu$ under $\varphi$.
\end{proof}

\section{Torus links and Characters}
\label{sec:mainthm}
Fix $r\geq 2$. In this section, we will analyze  $\la{sl}_r$ invariants of torus links $T(c,cp)$ with $1\leq c\leq r+1$, where each of the $c$ components is coloured with the irreducible module $L_r(n\Lambda_1)$.
\subsection{Number of components is $c\leq r$}
In what follows, we let $c\leq r$. Recalling \eqref{eqn:ourjones} we have
\begin{align*}
	J_r(c,pc;n)=	
	\sum_{\substack{\lambda: \lambda\vdash nc\\ \ell(\lambda)\leq c}} K_{\lambda, (n)^{c}}
	\cdot
	q^{\frac{p}{2}\kappa_{\lambda}}
	\cdot s_\lambda(q^{\frac{r-1}{2}},q^{\frac{r-3}{2}},\cdots,q^{\frac{1-r}{2}}).
\end{align*}

With $\lambda\vdash nc$ and $\ell(\lambda)\leq c$, we get:
\begin{align*}
	s_\lambda&(q^{\frac{r-1}{2}},q^{\frac{r-3}{2}},\cdots,q^{\frac{1-r}{2}})
	=\prod_{1\leq i<j\leq r}
	\frac{q^{\frac{1}{2} (\lambda_i-\lambda_j+j-i)}- q^{-\frac{1}{2} (\lambda_i-\lambda_j+j-i)}}
	{q^{\frac{1}{2} (j-i)}- q^{-\frac{1}{2} (j-i)}}\notag\\
	&=\prod_{1\leq i<j\leq c}
	\frac{q^{\frac{1}{2} (\lambda_i-\lambda_j+j-i)}- q^{-\frac{1}{2} (\lambda_i-\lambda_j+j-i)}}
	{q^{\frac{1}{2} (j-i)}- q^{-\frac{1}{2} (j-i)}}
	\prod_{\substack{1\leq i\leq c < j \leq r}}
	\frac{q^{\frac{1}{2} (\lambda_i+j-i)}- q^{-\frac{1}{2} (\lambda_i+j-i)}}
	{q^{\frac{1}{2} (j-i)}- q^{-\frac{1}{2} (j-i)}}
	\notag\\
	&=s_\lambda(q^{\frac{c-1}{2}},q^{\frac{c-3}{2}},\cdots, q^{\frac{1-c}{2}})
	\prod_{\substack{1\leq i\leq c < j \leq r}}
	\frac{q^{\frac{1}{2} (\lambda_i+j-i)}- q^{-\frac{1}{2} (\lambda_i+j-i)}}
	{q^{\frac{1}{2} (j-i)}- q^{-\frac{1}{2} (j-i)}}
	\notag\\
	&=q^{-\frac{nc(r-c)}{2}}s_\lambda(q^{\frac{c-1}{2}},q^{\frac{c-3}{2}},\cdots, q^{\frac{1-c}{2}})
	\prod_{\substack{1\leq i\leq c < j \leq r}}
	\frac{1-q^{\lambda_i+j-i}}
	{1-q^{j-i}}
\end{align*}

Let 
\begin{align*}
	f_{n,c,r,p}=q^{\frac{p}{2}\left(-n^2c+nc^2\right)+\frac{nc(r-c)}{2}}\prod_{1\leq i\leq c<j\leq r}(1-q^{j-i}).
\end{align*}

We then have:
\begin{align}
	&f_{n,c,r,p}J_r(c,pc;n)\notag\\
	&\,\,=	
	\sum_{\substack{\lambda: \lambda\vdash nc\\ \ell(\lambda)\leq c}} K_{\lambda, (n)^{c}}
	\cdot
	q^{\frac{p}{2}\left( \kappa_{\lambda}-\frac{1}{c}|\lambda|^2+c|\lambda| \right) }
	\cdot s_\lambda(q^{\frac{c-1}{2}},q^{\frac{r-3}{2}},\cdots,q^{\frac{1-c}{2}})\notag\\
	&\quad\quad\cdot
	\prod_{1\leq i\leq c<j\leq r}(1-q^{\lambda_i+j-i}).
	\label{eqn:Jc<=r}
\end{align}

We now define a set $W_n\subseteq P_c^+\cap Q_c$ which is in a natural bijection with the set of $\lambda$s appearing in this sum.
Suppose $\mu = a_1\Lambda_1+\cdots+a_{c-1}\Lambda_{c-1}\in P_{c}^+$ ($a_i\in \ZZ_{\geq 0}$ for all $i$)
such that $a_1+2a_2+\cdots+(c-1)a_{c-1}\leq n$. 
Further assume that $\mu\in P_c^+\cap Q_c$ which is equivalent to assuming that $a_1+2a_2+\cdots+(c-1)a_{c-1}$ is divisible by $c$.
Let $W_n\subseteq P_c^+\cap Q_c$ be the set of such weights.

Given $\mu=a_1\Lambda_1+\cdots+a_{c-1}\Lambda_{c-1}\in W_n$, choose
\begin{align*}
	a_c= n - \frac{a_1+2a_2+\cdots+(c-1)a_{c-1}}{c}\in \ZZ_{\geq 0}
\end{align*}
and consider the partition 
\begin{align*}
	\lambda = (a_1+a_2+\cdots+a_{c}, a_2+\cdots+a_c,\cdots, a_c,0,0,\cdots).
\end{align*}
It is clear that $\lambda\vdash nc$, $\ell(\lambda)\leq c$, $L_c(\mu)\cong L_c(\lambda)$.
This assignment $\mu\mapsto \lambda$ is easily seen to be a bijection from $W_n$ to the set of $\lambda$ that appear in the sum \eqref{eqn:Jc<=r}.
Moreover, we have
$$K_{\lambda,(n)^c}=\dim(L_c(\mu)_0)$$ from Proposition \ref{prop:dim0} and  $$\kappa_{\lambda}-\frac{1}{c}|\lambda|^2+c|\lambda|=(\mu,\mu+2\weylvec)$$
from \eqref{eqn:kappaa} and \eqref{eqn:thetakappa}.

We may thus rewrite \eqref{eqn:Jc<=r}:
\begin{align*}
	&f_{n,c,r,p}J_r(c,pc;n)\notag\\
	&\,\,=	
	\sum_{\mu\in W_n}
	\dim(L_c(\mu)_0)
	\cdot
	q^{\frac{p}{2}\left( \mu,\mu+2\weylvec\right) }
	\cdot s_\mu(q^{\frac{c-1}{2}},q^{\frac{r-3}{2}},\cdots,q^{\frac{1-c}{2}})\notag\\
	&\quad\quad\cdot
	\prod_{1\leq i\leq c<j\leq r}(1-q^{n+a_i+\cdots+a_{c-1} - \frac{a_1+2a_2+\cdots+(c-1)a_{c-1}}{c} +j-i}).
\end{align*}
Finally, we observe that
$W_n\subseteq W_{n+1}$ and that $\cup_{n\geq 0}W_n=Q_c\cap P_c^+$.

Now letting $n\rightarrow\infty$ and recalling \eqref{eqn:singletchar}, we easily deduce our first main theorem:
\begin{thm} Fix $r\geq c\geq 2$. Then,
	\begin{align*}
		\lim_{n\rightarrow\infty }
		&q^{\frac{p}{2}(-nc^2+nc^2)+\frac{nc(r-c)}{2}}J_{r}(c,pc;n)\notag\\
		&\,\,=\frac{1}{\prod_{1\leq i\leq c<j\leq r}(1-q^{j-i})}\frac{(q)_\infty^{c-1}}
		{\prod_{1\leq i<j\leq c}(1-q^{j-i})}
		\overline{\ch}(W^0(p)_{Q_c}).
	\end{align*}	
\end{thm}

\begin{rem}
	The $r=c=2$ case of this theorem is established in \cite{Haj} using skein theory. The $r=3$, $c=2$ case of this theorem appears in \cite{Yua-torus}, where the theory of $\mathrm{A}_2$ webs is used.
\end{rem}

\subsection{Number of components is $c=r+1$}

Fix $i$ with $0\leq i\leq r-1$ and suppose that $n\equiv i\,\,(\text{mod}\,\,r)$.
Recalling \eqref{eqn:ourjones}, we have
\begin{align*}
	J_r(r+1,p(r+1);n)=	\sum_{\substack{\lambda: \lambda\vdash n(r+1)\\ \ell(\lambda)\leq r}} K_{\lambda, (n)^{r+1}}
	\cdot
	q^{\frac{p}{2}\kappa_{\lambda}}
	\cdot s_\lambda(q^{\frac{r-1}{2}},q^{\frac{r-3}{2}},\cdots,q^{\frac{1-r}{2}}).
\end{align*}
Therefore,
\begin{align*}
	&q^{\frac{p}{2}\left(-\frac{n^2(r+1)^2}{r} + nr(r+1)\right)}J_r(r+1,p(r+1);n)\notag\\
	&\,\,=\sum_{\substack{\lambda: \lambda\vdash n(r+1)\\ \ell(\lambda)\leq r}} K_{\lambda, (n)^{r+1}}
	\cdot
	q^{\frac{p}{2}\left(\kappa_{\lambda}-\frac{1}{r}|\lambda|^2+r|\lambda|\right)}
	\cdot s_\lambda(q^{\frac{r-1}{2}},q^{\frac{r-3}{2}},\cdots,q^{\frac{1-r}{2}})\notag\\
	&\,\,=
	\left(\sum_{\substack{\lambda: \lambda\vdash n(r+1)\\ \ell(\lambda)\leq r \\ \lambda_r\geq n}} +
	\sum_{\substack{\lambda: \lambda\vdash n(r+1)\\ \ell(\lambda)\leq r\\ \lambda_r<n}} 
	\right)
	K_{\lambda, (n)^{r+1}}
	\cdot
	q^{\frac{p}{2}\left(\kappa_{\lambda}-\frac{1}{r}|\lambda|^2+r|\lambda|\right)}
	\cdot s_\lambda(q^{\frac{r-1}{2}},q^{\frac{r-3}{2}},\cdots,q^{\frac{1-r}{2}})\notag\\
	&= S_1(n)+S_2(n).
\end{align*}

We analyze the sums $S_1(n)$ and $S_2(n)$ separately.

Let us start with $S_1(n)$.  

Define  $C_n\subseteq P_r^+\cap( Q_r + i\Lambda_1)$ to be the set of weights
$\mu\in P_r^+\cap( Q_r + i\Lambda_1)$ such that $\mu = a_1\Lambda_1+\cdots+a_{r-1}\Lambda_{r-1}$ ($a_i\in\ZZ_{\geq 0}$ for all $i$) and $a_1+2a_2+\cdots+(r-1)a_{r-1}\leq n$.

Given a $\mu\in C_n$, 
start with the partition $\widehat{\mu}=(a_1+a_2+\cdots+a_{r-1},a_2+\cdots+a_{r-1},\cdots,a_{r-1},0,\dots)$ and adjoin $n+\frac{n-|\widehat{\mu}|}{r}$ many columns of height $r$ to the left of the Young diagram of $\widehat{\mu}$ to obtain a partition $\lambda$. 
The partition $\lambda$ has weight $nr+n$, $\ell(\lambda)\leq r$, and $\lambda_r\geq n$. 

It can be easily seen that this procedure is a bijection of $C_n$ with the set of $\lambda$ appearing in $S_1(n)$. Importantly, under the correspondence, $$\kappa_{\lambda}-\frac{1}{r}|\lambda|^2+r|\lambda|=(\mu,\mu+2\weylvec),$$ 
from \eqref{eqn:kappaa} and \eqref{eqn:thetakappa}.
Additionally, due to Proposition \ref{prop:r+1dim},
$$K_{\lambda,(n)^{r+1}}=\dim(L_r(\lambda))=\dim(L_r(\mu)).$$

Combining all of this we may now write the first sum as:
\begin{align*}
	S_1(n)=\sum_{\mu\in C_n}
	\dim(L_r(\mu))
	\cdot
	q^{\frac{p}{2}(\mu,\mu+2\weylvec)}
	\cdot s_\mu(q^{\frac{r-1}{2}},q^{\frac{r-3}{2}},\cdots,q^{\frac{1-r}{2}})
\end{align*}
Now we observe that 
$C_i\subseteq C_{i+r}\subseteq C_{i+2r}\subseteq\cdots$ with 
$\cup_{j\geq 0} C_{i+rj}=P_r^+\cap (Q_r+i\Lambda_1)$.
Comparing with \eqref{eqn:tripletchar},
\begin{align*}
\lim_{\substack{n \equiv i\,\,(\text{mod}\,\,r)\\n\rightarrow\infty}} S_1(n) = 	
\frac{(q)_\infty^{r-1}}
{\prod_{1\leq i<j \leq r}(1-q^{j-i})}
\overline{\ch}(W(p,\sqrt{p}\,i\Lambda_1)_{Q_r}).
\end{align*}

For the second sum $S_2(n)$, our aim is to show that it vanishes as $n\rightarrow \infty$.

So, let $\lambda\vdash n(r+1)$ with $\ell(\lambda)\leq r$ and $\lambda_r<n$ be an arbitrary partition that appears in $S_2(n)$.
Suppose that $a_i=\lambda_i-\lambda_{i+1}\in\ZZ_{\geq 0}$ for all $i$ (note that $\lambda_i=0$ for all $i>r$).
Since $\lambda_r=a_r<n$, we have:
\begin{align*}
	\sum_{1\leq i\leq r-1}ia_i = \lambda_1+\cdots+\lambda_{r-1}+\lambda_r-r\lambda_r>n.
\end{align*}
The lowest exponent of $q$ in the summand corrresponding to $\lambda$ in $S_2(n)$ is:
\begin{align*}
	\frac{p}{2}&\left(\kappa_{\lambda}-\frac{1}{r}|\lambda|^2+r|\lambda|\right)
	-\frac{1}{2}\sum_{1\leq i<j\leq r}\lambda_i-\lambda_j\notag\\
	&=\frac{p}{2}\left( \sum_{1\leq i\leq j\leq r-1}a_ia_j\left(\min(i,j)-\frac{ij}{r}\right) 
	+\sum_{1\leq i\leq r-1}i(r-i)a_i\right)-\frac{1}{2}\sum_{1\leq i\leq r-1}ia_i\notag\\
	&\geq 
	\frac{p}{2} \sum_{1\leq i\leq r-1}a_i^2\left(i-\frac{i^2}{r}\right) 
	+\frac{p}{2}\sum_{1\leq i\leq r-1}ia_i-\frac{1}{2}\sum_{1\leq i\leq r-1}ia_i\notag\\
	&\geq \frac{p}{2r} \sum_{1\leq i\leq r-1}ia_i
	+\frac{p-1}{2}\sum_{1\leq i\leq r-1}ia_i\notag\\
	&\geq \left( \frac{p}{2r} +\frac{p-1}{2}\right) n.
\end{align*}	
Thus, for any fixed $r\geq 2$ and $p\geq 1$, this approaches infinity as $n\rightarrow \infty$.

We now conclude our second main theorem:

\begin{thm} Fix $r\geq 2$. Let $0\leq i\leq r-1$. Then,
\begin{align*}
&\lim_{\substack{n \equiv i\,\,(\text{mod}\,\,r)\\n\rightarrow\infty}}q^{\frac{p}{2}\left(-\frac{n^2(r+1)^2}{r} + nr(r+1)\right)}J_r(r+1,p(r+1);n) \notag\\
&\quad = 	
\frac{(q)_\infty^{r-1}}
{\prod_{1\leq i<j \leq r}(1-q^{j-i})}
\overline{\ch}(W(p,\sqrt{p}\,i\Lambda_1)_{Q_r}).
\end{align*}
\end{thm}


\providecommand{\oldpreprint}[2]{\textsf{arXiv:\mbox{#2}/#1}}\providecommand{\preprint}[2]{\textsf{arXiv:#1
		[\mbox{#2}]}}

\end{document}